\documentclass{amsart}

\usepackage{amsmath}
\usepackage{amsxtra}
\usepackage{amscd}
\usepackage{amsthm}
\usepackage{amsfonts}
\usepackage{amssymb}
\usepackage{color}
\usepackage{eucal}
\usepackage[all]{xy}
\usepackage{graphicx}
\usepackage{mathabx}
\usepackage{hyperref}

\newtheorem{theorem}{Theorem}[section]
\newtheorem*{theorem*}{Theorem}

\newtheorem{lemma}[theorem]{Lemma}
\newtheorem{proposition}[theorem]{Proposition}
\newtheorem*{proposition*}{Proposition}

\newtheorem{conjecture}[theorem]{Conjecture}
\newtheorem*{conjecture*}{Conjecture}
\newtheorem{thm-dfn}[theorem]{Theorem-Definition}
\newtheorem{claim}[theorem]{Claim}

\theoremstyle{definition}

\newtheorem{remark}[theorem]{Remark}
\newtheorem{remarks}[theorem]{Remarks}

\numberwithin{equation}{section}

\newcommand{\quash}[1]{}  

\newcommand{\frakb}{{\mathfrak b}}

\newcommand{\frakg}{{\mathfrak g}}

\newcommand{\frakl}{{\mathfrak l}}

\newcommand{\frakp}{{\mathfrak p}}

\newcommand{\bbD}{{\mathbb D}}

\newcommand{\bbF}{{\mathbb F}}

\newcommand{\bbQ}{{\mathbb Q}}

\newcommand{\bbZ}{{\mathbb Z}}

\newcommand{\calC}{{\mathcal C}}

\newcommand{\calF}{{\mathcal F}}
\newcommand{\calG}{{\mathcal G}}

\newcommand{\calL}{{\mathcal L}}
\newcommand{\calM}{{\mathcal M}}
\newcommand{\calN}{{\mathcal N}}
\newcommand{\calO}{{\mathcal O}}

\newcommand{\calS}{{\mathcal S}}

\newcommand{\calX}{{\mathcal X}}

\newcommand{\wt}{\widetilde}
\newcommand{\ul}{\underline}

\begin{document}

\title{On parabolic restriction of perverse sheaves}
\author{R. Bezrukavnikov and A. Yom Din}

\maketitle

\centerline{\em{To Professor Masaki Kashiwara with admiration on the occasion of his 70th birthday}}

\tableofcontents

\section{Introduction}

\subsection{Main result}

Let $G$ be a connected reductive group over a finite field $\bbF_q$. Let $P \subset G$ be a parabolic, with Levi quotient $L = P / U$. One has an adjoint pair of functors of \emph{parabolic induction} and \emph{parabolic restriction} $$ pind^P_G : D(L \backslash L) \rightleftarrows D(G \backslash G) : pres^G_P.$$ Here, by $D(H \backslash H)$ we mean the equivariant derived category of constructible $\bar{\bbQ}_{\ell}$-sheaves on $H$, equivariant with respect to the adjoint action of $H$ on itself. 

 Grothendieck's \emph{sheaf-to-function} correspondence
relates these functors to the operators on the space of functions on the corresponding Chevalley groups which send the character of a representation
to the character of its parabolic  induction and parabolic restriction, respectively. 


\medskip

Recall the subcategory $D(G \backslash G)^{\heartsuit} \subset D(G \backslash G)$ of \emph{perverse sheaves}. Lusztig introduced a subcategory $D(G \backslash G)^{\textnormal{ch}} \subset D(G \backslash G)$, the subcategory of \emph{character sheaves}, such that $D(G \backslash G)^{\heartsuit , \textnormal{ch}} := D(G \backslash G)^{\heartsuit} \cap D(G \backslash G)^{\textnormal{ch}}$ matches tightly under the above-mentioned dictionaries with representations in $\textnormal{Rep} (G(\bbF_q))$. 
The definition of this subcategory is geometric, it applies to reductive algebraic groups over an arbitrary field.

\medskip

The functors of parabolic induction and parabolic restriction take character sheaves to character sheaves. One fundamental result of Lusztig is the following:

\begin{theorem*}[\cite{Lu2}]
	The functors $pres^G_P$ and $pind^P_G$, when restricted to the subcategories of character sheaves, are $t$-exact; In other words, they transform perverse character sheaves into perverse character sheaves.
\end{theorem*}

The main theorem of this note is the following generalization of the above theorem:

\begin{theorem*}[Theorem \ref{thm t-exactness}]
	The functors $pres^G_P$ and $pind^P_G$ are $t$-exact; In other words, they transform perverse sheaves into perverse sheaves.
\end{theorem*}

Thus, we extend $t$-exactness from character sheaves to all adjoint-equivariant sheaves. 

\medskip

Our method works equally well for $\ell$-adic sheaves, constructible sheaves in the complex-analytic topology and (not necessarily holonomic) $D$-modules (see \S\ref{sec notations sheaf} for the list of ``sheaf-theoretic settings").

\medskip

Our method also works in the close setting of  $G$-equivariant sheaves on the Lie algebra $\frakg$ (see \S\ref{sec Lie alg}). Let us remark that, in the case of $G$-equivariant $D$-modules on $\frakg$, the theorem was established earlier in \cite{Gu}. However, in \emph{loc. cit.} the tools of Fourier transform (which is not available in the group case) and singular support (which has been defined for $\ell$-adic sheaves only recently \cite{Beil}, \cite{Sai}) are used. Sam Gunningham has informed us that he also expects his methods to generalize to the group case.

\medskip

Our method is based on a known expression for the dimension of the intersection of a conjugacy class with a Borel subgroup (see \S\ref{sec conj cl}, and in particular proposition \ref{prop est for borel}), it also relies on Braden's Theorem on the behavior of hyperbolic restriction with respect to Verdier duality \cite{Braden}, \cite{DrGa2}.

\subsection{A conjecture}

In section \ref{sec conj} we recall the fundamental \emph{Harish-Chandra transform} $$ \textnormal{HC}_* : D(G \backslash G) \to D(G \backslash (G/U \times G/U) / T)$$ and the \emph{long intertwining transform} $$ R_! : D(G \backslash (G/U \times G/U) / T) \to D(G \backslash (G/U \times G/U^-) / T).$$ By works \cite{BeFiOs} and \cite{ChYo}, the composition $R_! \circ \textnormal{HC}_*$ is $t$-exact when applied to character sheaves. We propose the following conjectural generalization:

\begin{conjecture*}[Conjecture \ref{conj conj}]
	The functor $R_! \circ \textnormal{HC}_*$ is $t$-exact.
\end{conjecture*}

We also explain how the main theorem of this note (in the Borel case) is an evidence towards this conjecture.

\subsection{Two applications}

In the work \cite{GaLo}, which deals with perverse sheaves on a torus $T$, a $t$-exact and conservative Mellin transform $$ \calM_* : D(T) \to D^b_{coh} (\calC(T))$$ is introduced, where $\calC(T)$ is, roughly, the space of tame local systems of rank $1$ on $T$. The Euler characteristic of a perverse sheaf $\calF \in D(T)^{\heartsuit}$ is interpreted as the generic rank of $\calM_* (\calF)$, and hence in particular is non-negative. A closely related fact is that the convolution of a perverse sheaf $\calF \in D(T)^{\heartsuit}$ with $\calL_{\chi} \in D(T)^{\heartsuit}$, the perverse local system of rank $1$ corresponding to $\chi \in \calC(T)$, is perverse for generic $\chi$: in fact $\calF*\calL_{\chi}=V_\calF(\chi) \otimes \calL_{\chi}$ where the vector space $V_\calF(\chi)$
is the fiber of $ \calM_* (\calF)$ at the point $\chi$.

\medskip

We generalize the above two properties replacing the torus by a general reductive group.

\begin{theorem*}[Theorem \ref{thm euler}]
	Let $\calF \in D(G \backslash G)^{\heartsuit}$, i.e. $\calF$ is a perverse sheaf on $G$, equivariant with respect to the adjoint $G$-action. Then the Euler characteristic of $\calF$ is non-negative.
\end{theorem*}

We prove this theorem by reducing to the torus case, via parabolic restriction. The result is, as far as we know, new in the $\ell$-adic setting as well as in the holonomic $D$-module setting, and we refer the reader to remark \ref{rem history euler} for its history and previously known cases.

\begin{proposition*}[Proposition \ref{prop gen perv}]
	Let $\calF \in D(G \backslash G)^{\heartsuit}$ be a $G$-equivariant perverse sheaf on $G$ and let $\calG \in D(G \backslash G)^{\heartsuit}$ be a perverse character sheaf with generic central character (here ``generic" depends on $\calF$). Then $\calF * \calG$ is perverse (where $``*"$ denotes convolution on the group).
\end{proposition*}

This proposition is also proved by reducing to the torus case, via parabolic restriction and induction.

It is tempting to conjecture that Theorem \ref{thm euler} and Proposition \ref{prop gen perv} are related to the yet unknown "Mellin transform for reductive groups" which
would relate $D(G \backslash G)^{\heartsuit}$ to sheaves on a space parametrizing character sheaves, just as their special case established in \cite{GaLo}
is related to the functor $ \calM_*$.

\subsection{Acknowledgments}

We would like to thank Alexander Braverman, Victor Ginzburg, Sam Gunningham, George Lusztig and Kostiantyn Tolmachov for helpful correspondence.

\section{Notations and conventions}

\subsection{Group-theoretic notations}

We work over an algebraically closed ground field.

\medskip

We fix a connected reductive algebraic group $G$, and a parabolic $L := \frac{P}{U} \xleftarrow{\pi} P \to G$. Thus, we denote by $U$ the unipotent radical of $P$, by $L$ the Levi quotient of $P$, and by $\pi : P \to L$ the quotient map.

\medskip

For a connected affine algebraic group $H$, we always understand $H$ to act on itself via conjugation so $H$-invariant subvarieties and $H$-orbits are understood accordingly; as another example, $H \backslash H$ denotes the quotient stack of $H$ by the adjoint action of $H$. We use the following notations. For $h \in H$, we denote $$ \calO^H_h :=  \{ khk^{-1} \ : \ k \in H\}.$$ Given an integer $d \ge 0$, we denote $$ H^{(d)} := \{ h \in H \ | \ \dim \calO^H_h = d\}.$$ If $H$ acts on a variety $X$, we denote $$ X^h := \{ x \in X \ | \ hx = x\}.$$ If $H$ is reductive, we denote by $X_H$ the variety of Borels in $H$.

\subsection{Sheaf-theoretic notations}\label{sec notations sheaf}

We only consider stacks which are of the form $H\backslash X$ where $H$ is an affine algebraic group acting on a variety $X$. For a stack $\calX$, we denote by $D(\calX)$ the derived category of ``sheaves" on $\calX$, meaning any one of the following ``sheaf-theoretic contexts"\footnote{Except of section \ref{sec app}, where we do not consider the $D$-module setting.}:

\begin{enumerate}
	\item Non-holonomic $D$-module setting: The ground field is of characteristic zero, and we consider the unbounded derived category of all $D$-modules, such as in \cite{DrGa1} etc.
	\item Holonomic $D$-module setting: The ground field is of characteristic zero, and we consider the bounded derived category of holonomic $D$-modules.
	\item $\ell$-adic setting: We fix a prime $\ell$ different from the characteristic of the ground field, and consider the bounded derived category of constructible $\ell$-adic sheaves.
	\item Complex-analytic setting: The ground field is the field of complex numbers, and we consider the bounded derived category of sheaves in the complex-analytic topology, constructible w.r.t. finite algebraic stratifications.
\end{enumerate}

\medskip

In each setting, $D(\calX)$ admits a natural $t$-structure (the ``perverse" one). We denote by $D(\calX)^{=0} , D(\calX)^{\ge 0}, \text{etc.}$ the subcategories of objects concentrated in the specified cohomological degrees.

\medskip

For a smooth stack $\calX$, by a \emph{smooth} sheaf in $D(\calX)$ we will understand a locally constant sheaf in the sheaf-theoretic contexts $(3)$ and $(4)$, and a coherent $D$-module which is locally free as an $\calO$-module - of finite rank in the sheaf-theoretic context $(2)$ and of perhaps infinite rank in the sheaf-theoretic context $(1)$.

\medskip

For $\calF \in D(H \backslash X)$, we denote by $\ul{\calF} \in D(X)$ the corresponding sheaf, i.e. the result of applying to $\calF$ the $t$-exact forgetful functor $p^{\circ} = p^! [-\dim H]$ where $p : X \to H \backslash X$.

\medskip

We will prefer stating things with the $(\pi^! , \pi_*)$-versions (as opposed to $(\pi^* , \pi_!)$) of functors, because they work better in the sheaf-theoretic context $(1)$.

\section{Conjugacy classes and parabolics}\label{sec conj cl}

In this section, we will provide information on some dimensions involving conjugacy classes, which will be used in the proof of proposition \ref{prop left t-exactness}. We recall that we fix a parabolic $L \xleftarrow{\pi} P \to G$.

\begin{proposition}\label{prop est for borel}
	Let $g \in G$, and let $B \subset G$ be a Borel subgroup. Then $$ \dim (\calO^G_g \cap B) = \frac{1}{2} \dim \calO^G_g.$$
\end{proposition}

\begin{proof}
	One has $$ \dim \calO^G_g + \dim X_G^g= \dim X_G + \dim (\calO^G_g \cap B)$$ (because both sides represent the dimension of $\{ (h,x) \in \calO^G_g \times X_G \ | \ h x = x\}$), and thus the desired equality is easily seen to be equivalent to the equality $$ \dim Z_G (g) = \text{rank} (G) + 2 \dim X_G^g.$$ For the latter, see \cite[Chapter 6]{Hu} and refereneces therein (for attributions, consult \emph{loc. cit.}). In \S 6.17 of \emph{loc. cit.} the general statement is derived from that for unipotent $g$. In \S 6.8 of \emph{loc. cit.} the statement for unipotent $g$ is derived from a ``density condition". References for the verification of this condition are given in \S 6.9 of \emph{loc. cit.} (some low characteristic cases need to be treated independently).
\end{proof}

\begin{lemma}[{\cite[Proposition 1.2 (a)]{Lu1}}]\label{lem est}
	Let $g \in G$ and $\ell \in L$. Then\footnote{By convention, the inequality is understood to hold if $\calO^G_g \cap \pi^{-1} (\ell)$ is empty.} $$ \dim (\calO^G_g \cap \pi^{-1} (\ell)) \leq \frac{1}{2} \left( \dim \calO^G_g - \dim \calO^L_{\ell} \right).$$
\end{lemma}

\begin{proof}

	Let us provide a proof which assumes proposition \ref{prop est for borel}; See \cite[Proposition 1.2 (a)]{Lu1} for a different proof.

	Since $$ \dim (\calO^G_g \cap \pi^{-1} (\ell)) = \dim (\calO^G_g \cap \pi^{-1} (\calO^L_{\ell})) - \dim (\calO^L_{\ell}),$$ the inequality to be established is equivalent to $$ \dim (\calO^G_g \cap \pi^{-1}(\calO^L_{\ell})) \overset{?}{\leq} \frac{1}{2} \left( \dim (\calO^G_g) + \dim (\calO^L_{\ell}) \right).$$
	
	Let us fix a Borel subgroup $B \subset L$. By considering the variety $$ Z_1 = \{ (h , x) \in (\calO^G_g \cap \pi^{-1} (\calO^L_{\ell})) \times X_L \ | \ \pi (h)  x = x \}$$ and its two projections, we see that \begin{equation}\label{eq prf est 1} \dim (\calO^G_g \cap \pi^{-1} (\calO^L_{\ell})) + \dim X_L^{\ell} = \dim Z_1 = \dim (X_L) + \dim (\calO^G_g \cap \pi^{-1} (\calO^L_{\ell} \cap B)). \end{equation}
	
	Similarly, by considering the variety $$ Z_2 = \{ (h,x) \in \calO^L_{\ell} \times X_L \ | \ h x = x\}$$ and its two projections, we see that \begin{equation}\label{eq prf est 2} \dim \calO^L_{\ell} + \dim X_L^{\ell} = \dim Z_2 = \dim X_L + \dim (\calO^L_{\ell} \cap B). \end{equation} We thus have: $$ \dim (\calO^G_g \cap \pi^{-1} (\calO^L_{\ell})) \overset{\text{eq. }\ref{eq prf est 1}}{=} \dim (\calO^G_g \cap \pi^{-1} (\calO^L_{\ell} \cap B)) + \left( \dim X_L - \dim X_L^{\ell} \right) \overset{\text{eq. }\ref{eq prf est 2}}{=} $$ $$ = \dim (\calO^G_g \cap \pi^{-1} (\calO^L_{\ell} \cap B)) + \left( \dim \calO^L_{\ell} - \dim (\calO^L_{\ell} \cap B) \right) \overset{\text{prop. }\ref{prop est for borel}}{=}$$ $$ = \dim (\calO^G_g \cap \pi^{-1} (\calO^L_{\ell} \cap B)) + \frac{1}{2} \dim \calO^L_{\ell} \leq \dim (\calO^G_g \cap \pi^{-1} (B)) + \frac{1}{2} \dim \calO^L_{\ell} \overset{\text{prop. }\ref{prop est for borel}}{=} $$ $$=  \frac{1}{2} \dim \calO^G_g + \frac{1}{2} \dim \calO^L_{\ell}.$$ 
\end{proof}

\begin{proposition}\label{prop main est}
	Let $d,e \ge 0$, and let $V \subset G^{(d)}$ be a $G$-invariant subvariety. Then for every $\ell \in L^{(e)}$ one has\footnote{By convention, the inequality is understood to hold if $V \cap \pi^{-1} (\ell)$ is empty.} $$ \dim (V \cap \pi^{-1} (\ell)) \leq \dim V - \dim (V \cap \pi^{-1} (L^{(e)})).$$
\end{proposition}

\begin{proof}
	For any given $\ell \in L^{(e)}$, there are only finitely many $G$-orbits $\calO \subset V$ for which $\calO \cap \pi^{-1} (\ell) \neq \emptyset$. Hence, from lemma \ref{lem est} we deduce \begin{equation}\label{ineq V 0} \dim (V \cap \pi^{-1} (\ell)) \leq \frac{1}{2} (d - e). \end{equation} Thus, in order to establish the desired inequality, it is enough to establish \begin{equation}\label{ineq V 1} \frac{1}{2} (d-e) \overset{?}{\leq} \dim V - \dim (V \cap \pi^{-1} (L^{(e)})). \end{equation} Let us denote by $Z$ the closure inside $L^{(e)}$ of $\pi (V \cap \pi^{-1} (L^{(e)}))$. By inequality \ref{ineq V 0} we have $$ \dim (V \cap \pi^{-1} (L^{(e)})) \leq \dim Z + \frac{1}{2} (d-e),$$ and thus inequality \ref{ineq V 1} will follow if we establish that \begin{equation} \dim Z -e \overset{?}{\leq} \dim V - d.\end{equation} However, this inequality is clear, since the right hand side is equal to the dimension of the closure of $\chi (V) \subset W \backslash \backslash T$, where $\chi : G \to W \backslash \backslash T$ is the usual "characteristic" map, while the left hand side is equal to the dimension of a subvariety of this closure.
\end{proof}

\section{Parabolic restriction and parabolic induction}

In this section we recall basic facts about parabolic restriction and parabolic induction of adjoint-equivariant sheaves. We recall that we fix a parabolic $L \xleftarrow{\pi} P \to G$.

\subsection{Parabolic restriction}

Recall that the \emph{parabolic restriction} functor $$pres^G_P : D(G \backslash G) \to D(L \backslash L)$$ is given as $q_* p^!$ where \begin{equation}\label{eq res ind corres}
\xymatrix{ & P \backslash P \ar[dl]_{p} \ar[dr]^{q} & \\G \backslash G & & L \backslash L} \end{equation} and the actions of the groups on themselves are via conjugation.

Let us give two non-equivariant descriptions of parabolic restriction, which we will use later.

\begin{remark}\label{remark non-equi description}
	Consider the correspondence \begin{equation}\label{eq unequi corres 1} \xymatrix{ & P \ar[dl]_{\iota} \ar[dr]^{\pi} & \\ G & & L}.\end{equation} Then one has a $2$-commutative diagram $$ \xymatrix{ D(G \backslash G) \ar[r]^{pres^G_P} \ar[d] & D(L \backslash L) \ar[d]\\ D(G) \ar[r]^{\pi_* \iota^!} & D(L)} $$ where the vertical arrows are the $t$-exact forgetful functors.
\end{remark}

\begin{remark}\label{remark non-equi description 2}
	Consider the correspondence $$ \xymatrix{ & (G \times G/P)_1 \ar[ld]_{\ul{p}} \ar[rd]^{\wt{q}} & & \\ G & & (\frac{G/U \times G/U}{L})_1 & L \ar[l]_(0.25){i}} $$ where $$ (G \times G/P)_1 := \{ (g , xP) \in G \times G/P \ | \ x^{-1} gx \in P\},$$ $ \frac{G/U \times G/U}{L}$ denotes the quotient by the diagonal right action, and $$ (\frac{G / U \times G/U}{L})_1 := \{ (xU , yU) \in \frac{G/U \times G/U}{L} \ | \ x^{-1} y \in P\}.$$ The maps are $$ \ul{p} (g,xP) = g , \ \wt{q} (g,xP) = (xU , gxU) , \ i(\ell) = (U , \ell U).$$ Then one has a $2$-commutative diagram $$ \xymatrix{ D(G \backslash G) \ar[r]^{pres^G_P} \ar[d] & D(L \backslash L) \ar[d]\\ D(G) \ar[r]^{i^! \wt{q}_* \ul{p}^!} & D(L)} $$ where the vertical arrows are the $t$-exact forgetful functors.
\end{remark}

\subsection{Parabolic induction}

Recall that $pres^G_P$ admits a left adjoint, denoted $$pind^P_G : D(L \backslash L) \to D(G \backslash G),$$ given by $p_* q^!$ in terms of the correspondence \ref{eq res ind corres}. Notice that the functor $pind^P_G$ is Verdier self-dual (since it the composition of a pull-back w.r.t. a smooth morphism of relative dimension $0$ and a push-forward w.r.t. a proper morphism).

\quash{Let us, for completeness, give a non-equivariant description of parabolic induction:

\begin{remark}\label{remark non-equi description 3}
	Consider the correspondence $$ \xymatrix{ & (G \times G/U)_1 \ar[d]_{\rho} \ar[rdd]^{q^{\prime}} & \\ & (G \times G/P)_1 \ar[ld]_{\ul{p}} & \\ G & & L} $$ where $$ (G \times G/U)_1 := \{ (g , xU) \in G \times G/U \ | \ x^{-1} gx \in P\}.$$ The yet-unspecified maps are $$\rho(g , xU) = (g , xP) , \ q^{\prime} (g , xU) = \pi (x^{-1} gx).$$ Then one has a $2$-commutative diagram $$ \xymatrix{ D(L \backslash L) \ar[r]^{pind^P_G} \ar[d] & D(G \backslash G) \ar[d] \\ D(L) \ar[r]^{(q^{\prime})^! (\rho^{\circ})^{-1} \ul{p}_*} & D(G)} $$ where the vertical arrows are the $t$-exact forgetful functors and by $(\rho^{\circ})^{-1} \calM$ we mean the (unique up to a unique isomorphism) object $\calN$ for which $\rho^{\circ} \calN \cong \calM$ , if such exists.
\end{remark}}

\subsection{Braden's hyperbolic localization and second adjunction}

Let $P^- \subset G$ be a parabolic opposite to $P$. Notice that the Levi factors of $P$ and $P^-$ are then canonically identified (both isomorphic to $P \cap P^-$). Denoting by $\iota^- , \pi^-$ the arrows as in the diagram \ref{eq unequi corres 1} but with $P$ replaced by $P^-$, Braden's hyperbolic localization theorem yields (as noted in \cite[\S 0.2.1]{DrGa2}):

\begin{proposition}\label{prop Braden}
	One has an isomorphism of functors $$ \pi_! \circ \iota^* \cong \pi^-_* \circ (\iota^-)^! : \ D(G)^{G\text{-mon}} \to D(L)$$ where $D(G)^{G\text{-mon}} \subset D(G)$ denotes the full subcategory generated by the image of the forgetful functor $D(G \backslash G) \to D(G)$.
\end{proposition}

From this proposition one can deduce that the functors $pres^G_P$ and $pres^G_{P^-}$ are Verdier dual to each other, and from this one obtains:

\begin{theorem}[Second adjunction]\label{thm second adj}
	The functor $pres^G_{P^-}$ is left adjoint to the functor $pind^P_G$.
\end{theorem}

\section{Parabolic restriction and parabolic induction are $t$-exact}

In this section, we will prove the main result of this note, namely that parabolic restriction and parabolic induction of adjoint-equivariant sheaves are $t$-exact. We recall that we fix a parabolic $L \xleftarrow{\pi} P \to G$.

\subsection{Left $t$-exactness of parabolic restriction} In this subsection, we establish the following proposition, on the way to proving the main theorem:

\begin{proposition}\label{prop left t-exactness}
	The functor $pres^G_P$ is left $t$-exact.
\end{proposition}

\begin{proof}
	It is enough to establish that for any $0 \neq \calM \in D(G \backslash G)^{=0}$ one has $pres^G_P \calM \in D()^{\ge 0}$, and in the $D$-module setting one can assume that $\calM$ is coherent. We can find a non-empty open $\nu : V \hookrightarrow \text{supp} (\calM)$  which is connected smooth and $G$-invariant, and such that $\nu^! \calM$ is smooth (see \S\ref{sec notations sheaf} for what ``smooth" means). Moreover, by making $V$ smaller we can assume that $V \subset G^{(d)}$ for some $d \ge 0$.
	
	By Noetherian induction on the support of $\calM$ we reduce to the case $\calM = \nu_* \calS$, where $\nu : V \hookrightarrow G$ is as above and $\calS \in D(G \backslash V)^{=0}$ is smooth.
	
	\medskip
	
	By remark \ref{remark non-equi description} it suffices to show that $$ \pi_* \iota^! \ul{\calS} \in D()^{\ge 0}.$$ Denoting by $i_e : L^{(e)} \hookrightarrow L$ the inclusion, it is enough to show that $$ i_e^! \pi_* \iota^! \ul{\calS} \in D()^{\ge 0}$$ for all $e \ge 0$. By base change, this is the same as to establish $$ \wt{\pi}_* \wt{i}_e^! \ul{\calS} \in D()^{\ge 0},$$ where $$ \xymatrix{ & V \cap \pi^{-1} (L^{(e)}) \ar[ld]_{\wt{i}_e} \ar[rd]^{\wt{\pi}} & \\ V & & L^{(e)}}.$$ Now, proposition \ref{prop main est} states that the dimensions of fibers of $\wt{\pi}$ are no bigger than the difference in dimension between the target and source of $\wt{i}_e$. Hence, using lemmas \ref{lem cohom amp 1} and \ref{lem cohom amp 2}, the claim follows.
\end{proof}

We record here the following two lemmas which we have used in the proof above:

\begin{lemma}\label{lem cohom amp 1}
	Let $\pi : X \to Y$ be a morphism of varieties, with $Y$ connected smooth. Denote $\dim X \leq e$ and $\dim Y = d$. Let $\calM \in D(Y)^{=0}$  be smooth. Then $\pi^! (\calM) \in D(X)^{\ge d-e}$.
\end{lemma}

\begin{lemma}\label{lem cohom amp 2}
	Let $\pi : X \to Y$ be a morphism of varieties, with dimensions of fibers $\leq f$. Then Then $\pi_* (D(X)^{\ge 0}) \subset D(Y)^{\ge -f}$.
\end{lemma}

\subsection{$t$-exactness}

We can now prove the main theorem of this note:

\begin{theorem}\label{thm t-exactness}
	The functors $pres^G_P$ and $pind^P_G$ are $t$-exact.
\end{theorem}

\begin{proof}
	Proposition \ref{prop left t-exactness} states that $pres^G_P$ is left $t$-exact. Thus, $pind^P_G$ is right $t$-exact. Since $pind^P_G$ is Verdier self-dual we deduce that $pind^P_G$ is also left $t$-exact - this is clear in the sheaf-theoretic contexts $(2),(3),(4)$, and will be explained in a moment in the sheaf-theoretic context $(1)$. Finally, by theorem \ref{thm second adj}, we deduce, since we just showed that $pind^{P^-}_G$ is left $t$-exact, that the functor $pres^G_P$ is right $t$-exact.
	
	\medskip
	
	Let us now explain, in the sheaf-theoretic context $(1)$, why the right $t$-exactness of $pind^P_G$ implies its left $t$-exactness. It is enough to show that for $\calM \in D(L \backslash L)^{ = 0}$ which is coherent, one has $pind^P_G \calM \in D()^{\ge 0}$. Since $pind^P_G \cong \bbD \circ pind^P_G \circ \bbD$, it is enough to show that for every $i \ge 0$ one has $\bbD pind^P_G (H^{-i} (\bbD \calM)[i]) \in D()^{\ge 0}$. We will use \cite{Ra} as a reference. The holonomic defect of $H^{-i} (\bbD \calM)$ is $\leq i$. Since, by \cite[Theorem 2.5.1]{Ra}, the standard functors do not increase holonomic defect, the holonomic defect of $pind^P_G (H^{-i} (\bbD \calM) [i]) $ is $\leq i$. Hence, by \cite[Proposition 2.6.1]{Ra} and by the established right $t$-exactness of $pind^P_G$, we have $\bbD pind^P_G (H^{-i} (\bbD \calM)[i]) \in D()^{\ge 0}$.
\end{proof}

\begin{remark}
	In the sheaf-theoretic contexts $(2),(3),(4)$ we can avoid using $pind^P_G$ in order to establish the right $t$-exactness of $pres^G_P$ - it simply follows from the left $t$-exactness by Verdier duality, using proposition \ref{prop Braden}.
\end{remark}

\subsection{Purity and semisimplicity}

Let us notice in passing here that, when we consider the setting of a finite ground field and mixed $\ell$-adic sheaves, the functor $pres^G_P$ preserves complexes of weight $\ge w$ (as is clear from remark \ref{remark non-equi description}) and complexes of weight $\leq w$ (as is clear from proposition \ref{prop Braden}). Thus, $pres^G_P$ is pure (preserves complexes pure of weight $w$). This allows, in a standard way, to deduce that, now in one of our sheaf-theoretic contexts $(3),(4)$, the functor $pres^G_P$ preserves semisimplicity of complexes ``of geometric origin". Let us here also notice that $pind^P_G$ preserves semisimplicity of complexes of geometric origin - this follows from the decomposition theorem.

\subsection{The Lie algebra case}\label{sec Lie alg}

In the case of $G$-equivariant sheaves on the Lie algebra $\frakg$, one has analogous results.

\medskip

First of all, Lie algebra versions of all the propositions in section \ref{sec conj cl} hold. These are stated analogously and proved in the same way, once one has the basic input, which is proposition \ref{prop est for borel}. Thus, we want to see that for $x \in \frakg$ and a Borel $\frakb \subset \frakg$, one has $$ \dim (\calO^{\frakg}_x \cap \frakb) = \frac{1}{2} \dim \calO^{\frakg}_{x}$$ (where $\calO^{\frakg}_x := \{  \textnormal{Ad} (g) x \ : \ g\in G \} $). By arguing in the same manner as in \cite[\S 6.17]{Hu}, we reduce to the case when $x$ is nilpotent. This latter case is handled, for example, in \cite[Corollary 3.3.24]{ChGi} (for attributions, consult \emph{loc. cit.}).

\medskip

Then, everything is defined and proven similarly to the group case. For example, parabolic restriction $pres^G_P$ is given as $q_* p^!$ where \begin{equation*}
\xymatrix{ & P \backslash \frakp \ar[dl]_{p} \ar[dr]^{q} & \\G \backslash \frakg & & L \backslash \frakl}. \end{equation*} We obtain:

\begin{theorem}
	The functors $$ pres^G_P : D(G \backslash \frakg) \to D(L \backslash \frakl)$$ and $$ pind^P_G : D(L \backslash \frakl) \to D(G \backslash \frakg)$$ are $t$-exact.
\end{theorem}

Let us notice that in the sheaf-theoretic context of $D$-modules, this theorem was proven, by different methods, in \cite{Gu}.

\section{A conjecture about $t$-exactness related to the Harish-Chandra transform}\label{sec conj}

In this section we provide a conjecture which generalizes theorem \ref{thm t-exactness} in the Borel case. We exclude the non-holonomic $D$-module setting for simplicity. We fix a pair of opposite Borel subgroups $B , B^- \subset G$. We denote by $U , U^-$ their respective unipotent radicals and $T := B \cap B^-$, seen also as the Levi quotient of $B$ and of $B^-$.

\subsection{The Harish-Chandra transform}\label{sec HC transform}

Recall the \emph{Harish-Chandra transform} $$\textnormal{HC}_* : D(G \backslash G) \to D(G \backslash (G/U \times G/U) / T),$$ given by $q_* p^!$ where $$ \xymatrix{ & G \backslash (G \times G/B) \ar[rd]^q \ar[ld]_p & \\ G \backslash G & & G \backslash (G / U \times G/U) / T}.$$ Here $p(g_1 , g_2 B) = g_1$ and $q(g_1 , g_2 B) = (g_2 U , g_1 g_2 U)$.

\medskip

One has a closed embedding $$ i : B \backslash T \cong G \backslash (G / U \times G/U)_{1} / T \to G \backslash (G / U \times G/U) / T$$ (here $B$ acts on $T$ by projecting onto $T$ and then acting by conjugation) where $$ (G / U \times G/U)_1 := \{ (g_1 U , g_2 U) \in G / U \times G/U \ | \ g_1^{-1} g_2 \in B \}$$ and the identification is by $t \mapsto (U , tU)$. Let us also denote by $$ \pi : B \backslash T \to T \backslash T$$ the natural map.

\medskip

Similarly to remark \ref{remark non-equi description 2}, we have \begin{equation}
\label{eq conj 1} \pi_* \circ i^! \circ \textnormal{HC}_* \cong pres^G_B : \ D(G \backslash G) \to D(T \backslash T). \end{equation}

\subsection{The long intertwining transform}

Recall the \emph{long intertwining transform} $$R_! : D(G \backslash (G/U \times G/U) / T) \to D(G \backslash (G/U \times G/U^-) / T),$$ given by $s_! r^*$ where $$ \xymatrix{ & G \backslash (G/U \times G/U \times G/U^-)_{w_0 , \flat} / T \ar[rd]^s \ar[ld]_r & \\ G \backslash (G / U \times G/U) / T & & G \backslash (G / U \times G/U^-) / T }.$$ Here $$ (G / U \times G/U \times G/U^-)_{w_0 , \flat} := \{ (g_1 U , g_2 U , g_3 U^-) \in G/U \times G/U \times G/U^- \ | \ g_2^{-1} g_3 \in U U^-\}.$$ The map $r$ is by projecting to the first and second coordinates, while the map $s$ is by projecting to the first and third coordinates.

\medskip

One has an open embedding $$ j : T \backslash T \cong G \backslash (G/U \times G/U^-)_{w_0} / T \to G \backslash (G/U \times G/U^-) / T$$ where $$ (G/U \times G/U^-)_{w_0} :=  \{ (g_1 U, g_2 U^-) \in G/U \times G/U^- \  | \ g_1^{-1} g_2 \in BU^-\}$$ and the identification is by $t \mapsto (U , tU^-)$.

\begin{lemma}\label{lem open closed}
	One has $$ \pi_* \circ i^! \cong j^! \circ R_! : \ D(G \backslash (G/U \times G/U) / T) \to D(T \backslash T).$$
\end{lemma}

\begin{proof}
	
	Consider the diagram:
	
	$$\xymatrix{ & & T \backslash U^- T \ar[ld]_{\wt{j}} \ar[rd]^{\wt{s}} & \\ & G \backslash (G/U \times G/U \times G/U^-)_{w_0} / T \ar[rd]^s \ar[ld]_r & & T \backslash T \ar[ld]_j \\ G \backslash (G / U \times G/U) / T & & G \backslash (G/U \times G/U^-) / T & }.$$ Here $T$ acts on $U^- T$ by conjugation. The map $\wt{j}$ is given by $u^- t \mapsto (U , u^- t U , t U^-)$ and the map $\wt{s}$ is given by $u^- t \mapsto t$. The diamond is Cartesian. Let us also denote by $$\iota : T \backslash T \to T \backslash U^- T$$ the inclusion. By the contraction principle, one has $\wt{s}_! \cong \iota^!$.
	
	\medskip

	One now has: $$ j^! \circ R_! = j^! \circ s_! \circ r^* \cong \wt{s}_! \circ \wt{j}^! \circ r^* \cong \iota^! \circ \wt{j}^! \circ r^* \cong (r \circ \wt{j} \circ \iota)^! [-2 \dim U] = \ldots$$ Notice now that $r \circ \wt{j} \circ \iota = i \circ \rho$ where $\rho : T \backslash T \to B \backslash T$ is the natural map. Also, one has $\rho^! \cong \pi_* [2\dim U]$ (because $\pi$ is a $U \backslash \bullet$-torsor). We thus further obtain $$ \cdots \cong \pi_* \circ i^!$$ as desired.
\end{proof}

\subsection{The conjecture} We propose the following conjecture:

\begin{conjecture}\label{conj conj}
	The functor $$ R_! \circ \textnormal{HC}_* : D(G \backslash G) \to D(G \backslash (G/U \times G/U^-) / T)$$ is $t$-exact.
\end{conjecture}

\begin{remark}
	According to \cite[Corollary 3.4]{BeFiOs}, in the setting of holonomic $D$-modules, the conjecture holds when we restrict the domain to that of \emph{character sheaves}. In the $\ell$-adic setting, the same assertion can be deduced from \cite[Theorem 7.8]{ChYo}.
\end{remark}

The following proposition is an evidence toward conjecture \ref{conj conj}.

\begin{proposition}
	The functor $$ j^! \circ R_! \circ \textnormal{HC}_* :\  D(G \backslash G) \to D(T \backslash T)$$ is $t$-exact.
\end{proposition}

\begin{proof}
	This follows by combining lemma \ref{lem open closed}, relation \ref{eq conj 1}, and the theorem \ref{thm t-exactness}.
\end{proof}

\section{Two applications}\label{sec app}

\subsection{Application to Euler characteristic}

In this subsection, we assume that we are in the holonomic $D$-module, $\ell$-adic or complex-analytic settings (i.e., we exclude non-holonomic $D$-modules).

\medskip

Recall that for a variety $X$ and $\calG \in D(X)$, one defines the \emph{Euler characteristic} $\text{Eul} (\calG) \in \bbZ$ by: $$ \text{Eul} (\calG) := \dim p_* \calG = \sum_{i \in \bbZ} (-1)^i \dim H^i (p_* \calG)$$ where $p : X \to \bullet$.

\medskip

We would like to prove:

\begin{theorem}\label{thm euler}
	Let $\calF \in D(G \backslash G)^{= 0}$. Then $ \text{Eul} (\ul{\calF}) $ is non-negative.
\end{theorem}

\begin{remarks}\label{rem history euler}\
	\begin{enumerate}
		\item When $G$ is a torus, the result is proved in \cite{GaLo} in the $\ell$-adic setting, in \cite{LoSa} in the holonomic $D$-module setting \quash{(see also citation in \cite[Theorem 9.1]{GaLo})} and in \cite{FrKa} in the complex-analytic setting. We will assume the torus case when proving the general case.
	
		\item In the complex-analytic setting, the result is proved in \cite{Ki} where the conjectural statement of the result is attributed to M.~Kapranov.

		\item In the $\ell$-adic setting or the holonomic $D$-module setting (and thus also the complex-analytic setting), assuming in addition that $\calF$ is an intermediate extension from the regular semisimple locus, the result is proved in \cite{Br}.

		\item Thus, we provide a new proof of the result, which holds in the complex-analytic setting as well as in the formerly unestablished $\ell$-adic setting.
		
	\end{enumerate}
\end{remarks}

In order to prove theorem \ref{thm euler}, let us first establish the following claim:

\begin{claim}\label{clm euler the same}
	Let $\calF \in D(G \backslash G)$. Then $$ \text{Eul} (\ul{pres^G_B (\calF)}) = \text{Eul} (\ul{\calF})$$ (where $B \subset G$ is a Borel subgroup).
\end{claim}

\begin{proof}
	For the proof, we choose a splitting of $B \to T$, so a maximal torus $T \subset B$.
	
	\medskip
	
	Denote by $pr : G \times G/B \to G$ the projection. By lemma \ref{lem euler 1}, we have $$\text{Eul} (pr^! \ul{\calF}) = \text{Eul} (\omega_{G/B}) \cdot \text{Eul} (\ul{\calF}) = |W| \cdot \text{Eul} (\ul{\calF}).$$ With respect to the standard $G$-action on $G \times G/B$ (by $g(h , xB) = (ghg^{-1} , gx B)$), the $T$-fixed points are of the form $(t , wB)$ for $t \in T , w \in W$. In particular, they all lie in $(G \times G/B)_1$ (see remark \ref{remark non-equi description 2} for this notation). Hence, by lemma \ref{lem euler 2}, working in the notations of remark \ref{remark non-equi description 2}, we have $$ \text{Eul} (pr^! \ul{\calF}) = \text{Eul} (\ul{p}^! \ul{\calF}) = \text{Eul} (\wt{q}_* \ul{p}^! \ul{\calF}).$$ Now, with respect to the diagonal $G$-action on $(\frac{G/U \times G/U}{T})_1$, the $T$-fixed points consist of the disconnected union of copies of $T$: $(wU , wTU)$ for $w \in W$. By lemma \ref{lem euler 2} and $G$-equivariancy, we obtain: $$ \text{Eul} (\wt{q}_* \ul{p}^! \ul{\calF}) = |W| \cdot \text{Eul} (i^! \wt{q}_* \ul{p}^! \ul{\calF}) = |W| \cdot \text{Eul} (\ul{pres^G_B (\calF)}).$$
\end{proof}

Now we are ready to prove theorem \ref{thm euler}.

\begin{proof}[Proof (of theorem \ref{thm euler}).]
	Let $\calF \in D(G \backslash G)^{= 0}$. By claim \ref{clm euler the same}, we have $ \text{Eul} (\ul{\calF}) = \text{Eul} (\ul{pres^G_B (\calF)})$. Since by theorem \ref{thm t-exactness} one has $pres^G_B (\calF) \in D()^{ = 0}$, in view of the torus case (remark \ref{rem history euler}.(1)) the theorem follows.
\end{proof}

We record here the following two lemmas which we have used above:

\begin{lemma}\label{lem euler 1}
	Let $X,Y$ be varieties, and let $\calG \in D(X)$. Denote by $\text{pr} : X \times Y \to X$ the projection. Then $$ \text{Eul} (\text{pr}^! \calG) = \text{Eul}(\omega_Y) \cdot \text{Eul} (\calG).$$
\end{lemma}

\begin{lemma}\label{lem euler 2}
	Let $X$ be a variety equipped with an action of a torus $T$. Let $\calF \in D(T \backslash X)$. Denoting $j : X^T \hookrightarrow X$, we have $$ \text{Eul} (\ul{\calF}) = \text{Eul} (j^! \ul{\calF}).$$
\end{lemma}

\subsection{Application to generic perversity}

In this subsection, we restrict ourselves to the $\ell$-adic setting. We fix a Borel $B \subset G$ with Levi $T$.

\medskip

In \cite{GaLo}, a $\bar{\bbQ}_{\ell}$-scheme $\calC (T)$ is defined, whose $\bar{\bbQ}_{\ell}$-points are the isomorphism classes of tame local systems of rank $1$  on $T$. Let us denote by $\calL_{\chi}$ the perverse local system of rank $1$ on $T$ corresponding to $\chi \in \calC(T)(\bar{\bbQ}_{\ell})$. Recall also the notion of \emph{character sheaves} in $D(G \backslash G)$ with \emph{central character} $\chi$ (or $\calL_{\chi}$). 

\medskip

\begin{proposition}\label{prop gen perv}
	Let $\calF \in D(G \backslash G)^{= 0}$. For a generic $\chi \in \calC (T) (\bar{\bbQ}_{\ell})$ (depending on $\calF$), given a character sheaf $\calG \in D(G \backslash G)^{=0}$ with central character $\chi$, one has $$\calF * \calG \in D(G \backslash G)^{=0}$$ (where $``*"$ denotes the $!$-convolution on the group).
\end{proposition}

\begin{proof}
	
	\emph{The torus case:} We want to check that for $\calF \in D(T)^{=0}$ and a generic $\chi \in \calC(T)$, one has $\calF * \calL_{\chi} \in D(T)^{=0}$. Recall the \emph{Mellin transform} $$ \calM_* : D(T) \to D^b_{coh} (\calC(T))$$ from \cite{GaLo}. By \cite[Theorem 3.4.7]{GaLo}, it is enough to check that $\calM_* (\calF * \calL_{\chi} )$ sits in degree $0$. By \cite[Proposition 3.3.1 (f)]{GaLo}, one has $$ \calM_* (\calF * \calL_{\chi}) \cong \calM_* (\calF) \otimes \calM_* (\calL_{\chi})$$ (here the ``$\otimes$" is in the derived sense). Since $\calM_* (\calF)$ sits in degree zero and, by \cite[Theorem 3.4.3]{GaLo}, is generically locally free, and $\calM_* (\calL_{\chi}) \cong (i_{\chi})_* \bar{\bbQ}_{\ell}$ (where $i_{\chi}$ is the inclusion of the point correponding to $\chi$), we see that indeed that $\calM_* (\calF) \otimes \calM_* (\calL_{\chi})$ sits in degree zero for generic $\chi$. 
	
	\medskip
	
	\emph{The general case:} We will reduce the general case to the torus case. Recall the morphisms $i$ and $\pi$ from \S\ref{sec HC transform}. Also, recall the transform $$ \textnormal{CH} : D(G \backslash (G/U \times G/U) / T) \to D(G \backslash G)$$ left adjoint to $\textnormal{HC}_*$.
	
	\medskip
	
	Since $\chi$ is generic (so its stabilizer in $W$ is trivial), $\calG$ is a direct summand of $pind^B_G \calL_{\chi}$, and hence we may assume $\calG = pind^B_G \calL_{\chi}$.
	
	\medskip
	
	We have: $$ \calF * pind^B_G \calL_{\chi} \cong \textnormal{CH} (\textnormal{HC}_* (\calF) * \calL_{\chi})$$ (where the latter ``$*$" denotes $!$-convolution w.r.t. the right action of $T$ on $G \backslash (G/U \times G/U) / T$ given by $(xU,yU)*t = (xU , ytU)$). Since $\chi$ is generic (so its stabilizer in $W$ is trivial), we have $$\textnormal{HC}_* (\calF) * \calL_{\chi} = i_* i^! (\textnormal{HC}_* (\calF) * \calL_{\chi}) \cong i_* \pi^* \pi_* i^! (\textnormal{HC}_* (\calF) * \calL_{\chi}) \cong $$ $$ \cong i_* \pi^*( ( \pi_* i^! \textnormal{HC}_* (\calF)) * \calL_{\chi}) \cong i_* \pi^* (pres^G_B \calF * \calL_{\chi})$$ and thus $$ \calF * pind^B_G \calL_{\chi} \cong \textnormal{CH} (i_* \pi^* (pres^G_B \calF * \calL_{\chi})) \cong pind^B_G (pres^G_B \calF * \calL_{\chi}).$$ Now the claim is clear by the $t$-exactness of $pind^B_G$ and $pres^G_B$ and the torus case above.
\end{proof}

\end{document}